\documentclass{article}
\usepackage{amsmath}
\usepackage{amsfonts}
\usepackage{amsthm}
\usepackage{mathabx}
\usepackage{enumerate}

\newtheorem{thm}{Theorem}[section]
\newtheorem{prop}[thm]{Proposition}
\newtheorem{cor}[thm]{Corollary}
\newtheorem{defn}[thm]{Definition}
\newtheorem{lem}[thm]{Lemma}

\begin{document}

\title{Generalized Riordan Groups and Operators on Polynomials}

\author{Shaul Zemel\thanks{This work was carried out while I was working at the Technische Universit\"{a}t Darmstadt in Germany, and was partially supported by DFG grant BR-2163/4-1.}}

\maketitle


\section*{Introduction}

Sequences of polynomials appear in many branches of mathematics and physics. Examples of such sequences are, e.g., the ones named after Bernoulli, Euler, Hermite, Lagrange, Laguerre, and others. Many of them are known to be \emph{Sheffer sequences}, one of its equivalent defining properties is that the linear operator taking $p_{n}$ to $np_{n-1}$ commutes with the translation operators. The umbral calculus was invented, initially symbolically and afterwards more rigorously, in order to understand the combinatorics of such sequences of polynomials, and prove identities between them and between their coefficients. One feature of these sequences is that their coefficients may be seen as the entries of certain infinite lower triangular matrices called \emph{Riordan arrays}.

There is a vast literature on the subject, large parts of which are aimed at applications to special sequences of polynomials. We first mention \cite{[R]}, which provides a clear introduction to the basics of this theory. A very partial list of references, in which one may find ways to deduce interesting results from the Riordan (or Sheffer) property, consists of \cite{[LM]}, \cite{[Sp]}, and the references cited there, as well as the older reference \cite{[Sh]}. In these references the Sheffer sequences are related to some shift operators, Riordan arrays, recurrence relations, linear functionals on polynomials, and further algebraic and combinatorial objects.

In this paper, however, we are interested in a more algebraic approach to this theory. \cite{[SGWW]} considers the set of Riordan arrays as a group (coining the term as well). \cite{[WW]} considers the more general groups of Riordan arrays, which are related to the non-classical umbral calculi appearing also in Chapter 6 of \cite{[R]} (both references, and the papers cited therein, should be added to the list from the previous paragraph). We consider the set of all  graded sequences of polynomials (with no additional properties) as a group of infinite, lower-triangular matrices. We then present our point of view on Sheffer sequences and Riordan arrays in the generalized sense of \cite{[WW]}, in particular as a subgroup of this bigger group. Next, we relate several equivalent properties that these sequences have: The column vectors of their matrices give, using the appropriate weights, a geometric sequence of formal power series; Their weighted generating function has a succinct expression as a function of two variables; The associated derivative-like operator commutes with appropriate translation-like operators; They respect a certain product rule for linear functionals on polynomials; And their weighted dual basis takes the form of a geometric sequence.

The group of $W$-Sheffer sequences contains two natural subgroups: The group of $W$-binomial sequences (which are characterized by their values at 0), and the group of $W$-Appell sequences. Example of those, together with some of their properties, are also given in the references mentioned above. As abstract groups these are isomorphic to formal power series of valuation 1 with composition and formal power series of valuation 0 with multiplication respectively. The $W$-Sheffer sequences are their semi-direct product. This structure is the same for every weight. Now, our large group of all graded sequences admits an action on a certain set of operators. We prove that the $W$-Appell sequences form the stabilizer of one such operator, and the $W$-Riordan group is its normalizer in the larger group. As this action is transitive, this explains why the algebraic structure must indeed be independent of the weight. Note that unlike many authors, we do not assume (except in one example) that our base field is of characteristic 0.

\smallskip

This paper is divided to 5 sections. In Section \ref{RArray} we present the large group, generalized Riordan arrays, and the relation to special forms of power series in two variables. Section \ref{ShefSeq} defines Sheffer sequences (and their subgroups) via translation and derivation operators, and proves the equivalence to Riordan arrays. In Section \ref{LinFunc} we investigate linear functionals on polynomials, and prove the relation between Sheffer sequences and products on such functionals. Section \ref{GroupOp} introduces the action of sequences on operators, and shows that Sheffer sequences normalize a stabilizer in this action. Finally, Section \ref{ExRel} gives the classical examples in this language, and also proves a result about the intersection of a generalized Riordan group of some weight with the Appell subgroup of the generalized Riordan group of a different weight.

\smallskip

I am greatful to A.-M. Luz\'{o}n-Cordero for her willingness to read a preliminary version of this paper, as well as for her encouragement to send it to publication.

\section{Riordan Groups \label{RArray}}

Let $\mathbb{F}$ be a field. $\mathbb{F}[x]$ denotes the ring of polynomials in one indeterminate $x$, and $\mathbb{F}\ldbrack y \rdbrack$ is the ring of formal power series in one indeterminate $y$ over $\mathbb{F}$. The latter ring is a complete discrete valuation ring. The valuation $v(C)$ of a non-zero element $C(y)=\sum_{n=0}^{\infty}c_{n}y^{n}\in\mathbb{F}\ldbrack y \rdbrack$ is the minimal index $n$ such that $c_{n}\neq0$, and we define $v(0)=\infty$.

Let two sequences $\{p_{n}(x)\}_{n\in\mathbb{N}}$ and $\{q_{n}(x)\}_{n\in\mathbb{N}}$ of polynomials in $\mathbb{F}[x]$ be given. Their \emph{umbral composition} $\{r_{n}(x)\}_{n\in\mathbb{N}}$ is defined as follows: If \[p_{n}(x)=\sum_{k=0}^{d_{n}}a_{n,k}x^{k}\quad\mathrm{with\ }a_{n,k}\in\mathbb{F}\quad\mathrm{then}\quad r_{n}(x)=\sum_{k=0}^{d_{n}}a_{n,k}q_{k}(x)\] (where $d_{n}$ is the degree of $p_{n}$). We consider these sequences as infinite column vectors (with integer indices $\geq0$) over $\mathbb{F}[x]$, where the natural ``base point'' for these sequences is the sequence of monomials $m_{n}(x)=x^{n}$. We now have \begin{prop}
For any sequence $\{p_{n}(x)\}_{n\in\mathbb{N}}$ there exists a unique infinite matrix $A$, in which each row has only finitely many non-zero entries, taking the monomial sequence to $\{p_{n}(x)\}_{n\in\mathbb{N}}$. Conversely, every row-finite matrix arises in this way from a unique sequence of polynomials. The umbral composition of sequences corresponds to the matrix product. \label{seqmat}
\end{prop}

The proof is simple and straightforward. Note that the essential finiteness of all the rows makes the matrix product in Proposition \ref{seqmat} well-defined. In addition, this property is preserved in products: If $A$ and $B$ correspond to the sequences $\{p_{n}(x)\}_{n\in\mathbb{N}}$ and $\{q_{n}(x)\}_{n\in\mathbb{N}}$ respectively, and if $d_{n}$ and $e_{n}$ are the degrees of $p_{n}$ and $q_{n}$ respectively, then $a_{n,k}=0$ for all $k>d_{n}$ and $b_{k,l}=0$ for all $l>e_{k}$. Hence $(AB)_{n,l}=0$ wherever $l>\max\{e_{k}\}_{k=0}^{d_{n}}$, so that the $n$th row of $AB$ also contains just finitely many non-zero entries. This corresponds to the fact that the umbral composition of two sequences of polynomials is indeed a sequence of polynomials.

We shall restrict attention to sequences of polynomials which are \emph{graded}, i.e., in which $d_{n}=n$ for all $n$. For these sequences we have
\begin{prop}
The sequence $\{p_{n}(x)\}_{n\in\mathbb{N}}$ is graded if and only if the corresponding matrix $A$ is lower triangular and invertible. The set of such matrices forms a group $L=L_{\infty}(\mathbb{F})$ under matrix multiplication. \label{lowtrigrad}
\end{prop}

The proof is again simple and direct. It follows from Propositions \ref{seqmat} and \ref{lowtrigrad} that the umbral composition of two graded sequences is also graded.

\smallskip

Let $W=W(t)=\sum_{n=0}^{\infty}\frac{t^{n}}{w_{n}}$ be an element of $\mathbb{F}\ldbrack t \rdbrack$ in which $w_{0}=1$ and $w_{n}\neq0$ for every $n$, which we fix once and for all. Dually to multiplication of column vectors from the left, right multiplication by the infinite row vector whose $n$th entry is $\frac{y^{n}}{w_{n}}$ easily establishes the following
\begin{prop}
Infinite matrices $A$ over $\mathbb{F}$ are in one-to-one correspondence with sequences $\big\{C_{A_{k},W}(y)\big\}_{k=0}^{\infty}$ of elements of $\mathbb{F} \ldbrack y \rdbrack$. The matrix $A$ is row-finite if and only if $v(C_{A_{k},W})\to\infty$ as $k\to\infty$. In particular we have $A \in L$ if and only if $v(C_{A_{k},W})=k$ for every $k\in\mathbb{N}$. \label{seqpowser}
\end{prop}

The power series $C_{A_{k},W}(y)$ is written explicitly as $\sum_{n=0}^{\infty}a_{n,k}\frac{y^{n}}{w_{n}}$. We shall write just $C_{A_{k}}$ instead of $C_{A_{k},W}$ in case $W$ is clear from the context. If the choice of $A$ is also clear we may shorten $C_{A_{k}}$ further to just $C_{k}$.

We can combine the two operations from Propositions \ref{seqmat} and \ref{seqpowser} together, and obtain that for any row-finite matrix $A$ the two expressions $\sum_{k=0}^{\infty}C_{A_{k}}(y)x^{k}$ and $\sum_{n=0}^{\infty}p_{n}(x)\frac{y^{n}}{w_{n}}$ represent the same element of the ring $\mathbb{F}\ldbrack x,y \rdbrack$ of formal power series in two variables. The second presentation shows that this element lies, in fact, in the subring $\mathbb{F}[x]\ldbrack y \rdbrack$ of $\mathbb{F}\ldbrack x,y \rdbrack$. Note that the order of the variables in this subring is important, as elements of the ring $\mathbb{F}\ldbrack y \rdbrack[x]$ must satisfy the stronger requirement that the total degree in $x$ has to be bounded.

\smallskip

We recall that a power series $0 \neq C\in\mathbb{F} \ldbrack y \rdbrack$ divides another power series $D\in\mathbb{F} \ldbrack y \rdbrack$ is and only if $v(C) \leq v(D)$, and the quotient $\frac{D}{C}$ has valuation $v(D)-v(C)$. In particular the set $\mathbb{F} \ldbrack y \rdbrack^{\times}$ of invertible elements in $\mathbb{F} \ldbrack y \rdbrack$ consists precisely of the power series of valuation 0. We shall henceforth assume that all our sequences of polynomials are graded, unless explicitly stated otherwise. In particular, Proposition \ref{seqpowser} shows that the quotient $C_{A_{k+1},W}/C_{A_{k},W}$ is a well-defined element of valuation 1 in $\mathbb{F} \ldbrack y \rdbrack$ for any $A \in L$ and any $k\geq0$.

We now make the following
\begin{defn}
We say that $A \in L$ is a \emph{Riordan array of weight $W$}, or a
\emph{$W$-Riordan array}, if the series $\big\{C_{A_{k},W}(y)\big\})_{k=0}^{\infty}$ is geometric, i.e., if the quotient $w_{k+1}C_{A_{k+1},W}/w_{k}C_{A_{k},W}$, of valuation 1, is independent of $k$. \label{Riordan}
\end{defn}
Definition \ref{Riordan} is equivalent to Definition 2.1 of \cite{[WW]}, according to the following simple
\begin{prop}
The matrix $A \in L$ is a $W$-Riordan array if and only if the corresponding element in $\mathbb{F}[x]\ldbrack y \rdbrack$ is of the form
$\alpha(y)W\big(x\beta(y)\big)$ for power series $\alpha$ and $\beta$ in
$\mathbb{F}\ldbrack y \rdbrack$ with $v(\alpha)=0$ and $v(\beta)=1$.
\label{genfunc}
\end{prop}
The proof is immediate from the expansion in powers of $x$ since a sequence $\big\{C_{k}(y)\big\})_{k=0}^{\infty}$ with $v(C_{k})=k$ is geometric if and only if it is of the form $C_{k}(y)=\alpha(y)\beta(y)^{k}$ with $\alpha=C_{0}$ is of valuation 0 and where $\beta$ is the quotient $w_{k+1}C_{A_{k+1},W}/w_{k}C_{A_{k},W}$ (independently of $k$), whose valuation is 1. We prefer to call it $W$-Riordan, rather that generalized Riordan, in order to emphasize its dependence on $W$. This will be particularly important when one is interested in changing $W$, as in Section \ref{ExRel} below. We denote the set of $W$-Riordan arrays by $R_{W}$. We may say just a \emph{Riordan array} in case the choice of $W$ is clear.

\smallskip

In addition to the multiplicative group $\mathbb{F} \ldbrack y \rdbrack^{\times}$, the set $y\mathbb{F} \ldbrack y \rdbrack^{\times}$ of elements of valuation 1 form a (non-commutative) group, with identity $e(y)=y$, under composition of power series. It turns out more convenient to use the convention of opposite composition, in which the product of the power series $\beta$ and $\delta$ is $\delta\circ\beta$. This composition law extends to the case where $v(\delta)$ is arbitrary, and in particular $v(\delta)=0$, yielding an action of the latter group on the former.

Theorem 2.4 of \cite{[WW]} now states that
\begin{prop}
The set $R_{W}$ is a subgroup of $L$. Moreover, it is isomorphic to the semi-direct product in which $y\mathbb{F} \ldbrack y \rdbrack^{\times}$ operates on $\mathbb{F} \ldbrack y \rdbrack^{\times}$ as above. \label{Rgroup}
\end{prop}
The idea is that if the left multiplier in a product is the element of $R_{W}$ with the parameters $\alpha$ and $\beta$ then its operation on our row vector takes out the scalar $\alpha(y)$ and transforms the variable $y$ by $\beta(y)$. According to Proposition \ref{Rgroup}, we call $R_{W}$ the \emph{$W$-Riordan group}, or just the \emph{Riordan group} in case $W$ is clear from the context.

\section{Sheffer Sequences \label{ShefSeq}}

Let $D_{W}:\mathbb{F}[x]\to\mathbb{F}[x]$ be the ``weighted derivative'', i.e., the linear operator which takes $\frac{x^{n}}{w_{n}}$ to $\frac{x^{n-1}}{w_{n-1}}$ for each $n\geq0$ (where $\frac{1}{w_{-1}}$ is defined to be 0). More generally, let $\{p_{n}(x)\}_{n\in\mathbb{N}}$ be a (graded) sequence of polynomials. There exists a unique operator, which we denote $Q_{W}$ (or $Q$ if $W$ is clear from the context), sending $\frac{p_{n}}{w_{n}}$ to $\frac{p_{n-1}}{w_{n-1}}$ for every $n$. Given $h\in\mathbb{F}$ we define the \emph{$W$-translation in $h$} to be the element $T_{h,W}$ of $L$ taking $\frac{x^{n}}{w_{n}}$ to $\sum_{k=0}^{n}\frac{h^{n-k}}{w_{n-k}}\cdot\frac{x^{k}}{w_{k}}$. We now make the following
\begin{defn}
The sequence $\{p_{n}(x)\}_{n\in\mathbb{N}}$ is called a \emph{Sheffer sequence of weight $W$}, or a \emph{$W$-Sheffer sequence}, if the operator $Q_{W}$ commutes with all the $W$-translations. \label{Sheffer}
\end{defn}
\cite{[WW]} (and others, in special cases) define a Sheffer sequences to be a (graded) sequence of polynomials whose associated element $A \in L$ from Proposition \ref{seqmat} lie in $R_{W}$. We will see in Theorem \ref{comp} that these conditions are indeed equivalent. This is known for a long time in the exponential case of $c_{n}=n!$ (see \cite{[Sh]}), but, as far as I know, not in the general case. In case confusion as to the sequence of polynomials (or element $A \in L$) may arise, we may use the notation $Q_{A,W}$ or $Q_{A}$ for the appropriate operator $Q_{W}$ or $Q$.

Given a graded sequence $\{p_{n}(x)\}_{n\in\mathbb{N}}$, it is clear from the definition of $T_{h,W}$ that the coefficient of $x^{k}$ in $T_{h,W}(p_{n})$ is a polynomial in $h$ of degree $n-k$. The same assertion continues to hold in case we express $T_{h,W}(p_{n})$ using the basis $\{p_{k}\}_{k\in\mathbb{N}}$ of $\mathbb{F}[x]$. We may therefore write \[T_{h,W}\bigg(\frac{p_{n}}{w_{n}}\bigg)(x)=\sum_{k=0}^{n}\frac{d_{n,k,A,W}(h)}{w_{n-k}}\cdot\frac{p_{k}(x)}{w_{k}},\quad\mathrm{with}\quad d_{n,k,A,W}(t)\in\mathbb{F}[t],\] and the degree of $d_{n,k,A,W}$ is $n-k$. Our next assertion is
\begin{lem}
The sequence $\{p_{n}(x)\}_{n\in\mathbb{N}}$ is $W$-Sheffer if and only if the polynomials $d_{n,k,A,W}$ depend only on the difference $n-k$. \label{Shefbin}
\end{lem}

\begin{proof}
We compare the two compositions $T_{h,W}Q_{A,W}$ and $Q_{A,W}T_{h,W}$ by evaluating their actions on the polynomial $\frac{p_{n}}{w_{n}}$. The first composition gives \[T_{h,W}Q_{A,W}\bigg(\frac{p_{n}}{w_{n}}\bigg)(x)=T_{h,W}\bigg(\frac{p_{n-1}}{w_{n-1}}\bigg)(x)=
\sum_{k=0}^{n-1}\frac{d_{n-1,k,A,W}(h)}{w_{n-1-k}}\cdot\frac{p_{k}(x)}{w_{k}},\]  while the other one yields \[\sum_{l=0}^{n}\frac{d_{n,l,A,W}(h)}{w_{n-l}}Q_{A,W}\bigg(\frac{p_{l}}{w_{l}}\bigg)(x)
\stackrel{l=k+1}{=}\sum_{k=0}^{n-1}\frac{d_{n,k+1,A,W}(h)}{w_{n-k-1}}\cdot\frac{p_{k}(x)}{w_{k}}\] (where we may omit the term with $l=0$ and $k=-1$ since $Q_{A,W}$ annihilates it). Since $\big\{\frac{p_{k}(x)}{w_{k}}\big\}_{k\in\mathbb{N}}$ is a basis for $\mathbb{F}[x]$, the two sides coincide if and only if the equality $d_{n-1,k,A,W}=d_{n,k+1,A,W}$ holds for every $k<n$. This proves the lemma.
\end{proof}

In the language of the elements of $\mathbb{F}[x]\ldbrack y \rdbrack\subseteq\mathbb{F}\ldbrack x,y \rdbrack$ from above, Lemma \ref{Shefbin} yields
\begin{prop}
The sequence $\{p_{n}(x)\}_{n\in\mathbb{N}}$ is $W$-Sheffer if and only if the
operator $T_{h,W}$ multiplies the corresponding element of
$\mathbb{F}[x]\ldbrack y \rdbrack$ by a power series in $y$ (and $h$),
independently of $x$, for every $h\in\mathbb{F}$. \label{Shefmul}
\end{prop}

\begin{proof}
Multiplication by the usual row vector, $T_{h,W}\big(\sum_{n=0}^{\infty}p_{n}\frac{y^{n}}{w_{n}}\big)(x)$ is seen to equal
\[\sum_{n=0}^{\infty}\sum_{k=0}^{n}y^{n}\frac{d_{n,k,A,W}(h)}{w_{n-k}}\cdot\frac{p_{k}(x)}{w_{k}}
\stackrel{n=k+l}{=}\sum_{k=0}^{\infty}\sum_{l=0}^{\infty}\frac{d_{k+l,k,A,W}(h)y^{l}}{w_{l}}\cdot\frac{p_{k}(x)y^{k}}{w_{k}}.\] If the polynomial $d_{n,k,A,W}$ (of degree $n-k$) is a polynomial $d_{n-k}$, depending only on the difference $n-k$, then the sum over $l$ yields the element $\sum_{l=0}^{\infty}d_{l}(h)\frac{y^{l}}{w_{l}}$ of $\mathbb{F}[h]\ldbrack y \rdbrack\subseteq\mathbb{F}\ldbrack h,y \rdbrack$, independently of $k$. Conversely, if $T_{h,W}$ multiplies our expression by some element of $\mathbb{F}\ldbrack h,y \rdbrack$, which we write as $\sum_{l=0}^{\infty}\frac{d_{l}(h)y^{l}}{w_{l}}$, then we compare this product with the expression written above. The linear independence of the the expressions $\frac{p_{k}(x)}{w_{k}}y^{m}$ with $k$ and $m$ from $\mathbb{N}$ shows that $d_{k+l,k,A,W}(h)$ must coincide with $d_{l}(h)$, depending only on $l$, for every $k$ and $l$. The assertion now follows from Lemma \ref{Shefbin}. This completes the proof of the proposition.
\end{proof}

We can now relate Definitions \ref{Riordan} and \ref{Sheffer} in
\begin{thm}
Let a weight $W=\sum_{n=0}^{\infty}w_{n}t^{n}$ and a graded sequence of polynomials $\{p_{n}(x)\}_{n\in\mathbb{N}}$, with the associated matrix $A \in L$, be given. Then the sequence $\{p_{n}(x)\}_{n\in\mathbb{N}}$ is $W$-Sheffer if and only if the matrix $A$ lies in $R_{W}$. \label{comp}
\end{thm}

\begin{proof}
If $A \in R_{W}$ then Proposition \ref{genfunc} allows us to write its associated element of $\mathbb{F}[x]\ldbrack y \rdbrack$ as $\alpha(y)W\big(x\beta(y)\big)$, where $v(\alpha)=0$ and $v(\beta)=1$.
Expanding $W$, the image under $T_{h,W}$ is evaluated as \[\alpha(y)\sum_{n=0}^{\infty}\beta(y)^{n}\sum_{k=0}^{n}\frac{h^{n-k}x^{k}}{w_{n-k}x_{k}}\stackrel{n=k+l}{=}
\alpha(y)\sum_{l=0}^{\infty}\frac{h^{l}\beta(y)^{l}}{w_{l}}\sum_{k=0}^{\infty}\frac{x^{k}\beta(y)^{k}}{w_{k}},\] which is the original expression multiplied by $W\big(h\beta(y)\big)$. This direction thus follows from Proposition \ref{Shefmul}. Conversely, we write this power series using the general formula $\sum_{l=0}^{\infty}C_{A_{l},W}(y)x^{l}$. Applying $T_{h,W}$ yields \[\sum_{l=0}^{\infty}C_{A_{l}}(y)\sum_{k=0}^{l}\frac{w_{l}h^{l-k}x^{k}}{w_{l-k}w_{k}}=
\sum_{k=0}^{\infty}C_{A_{k}}(y)x^{k}+\frac{h}{w_{1}}\sum_{k=0}^{\infty}C_{A_{k+1}}(y)\frac{w_{k+1}x^{k}}{w_{k}}+O(h^{2}),\] where the $O(h^{2})$ means a power series in $x$, $y$, and $h$ which contains only powers of $h$ which are at least 2. Comparing this with the product of the original series $\sum_{k=0}^{\infty}C_{A_{k},W}(y)x^{k}$ multiplied by some element of $\mathbb{F}\ldbrack h,y \rdbrack$, which we write as $\gamma(y)+\frac{h}{w_{1}}\beta(y)+O(h^{2})$ using some elements $\beta$ and $\gamma$ of $\mathbb{F}\ldbrack y \rdbrack$, yields the equality $\gamma(y)=1$ from the constant terms (in $h$), as well as the equality
$\sum_{k=0}^{\infty}C_{A_{k+1}}(y)\frac{w_{k+1}x^{k}}{w_{k}}=\beta(y)\sum_{k=0}^{\infty}C_{A_{k}}(y)x^{k}$ from the terms linear in $h$. But this implies that the quotient $\frac{w_{k+1}C_{k+1}}{w_{k}C_{k}}$ equals $\beta$ for all $k$, so that $A \in R_{W}$ by Definition \ref{Riordan}. This completes the proof of the theorem.
\end{proof}

The proofs of Proposition \ref{Shefmul} and Theorem \ref{comp} also yield the
following
\begin{cor}
Let $\{p_{n}(x)\}_{n\in\mathbb{N}}$ be a Sheffer sequence. Then the parameters  $\alpha$ and $\beta$ of $\mathbb{F}\ldbrack y \rdbrack$ from Proposition \ref{genfunc} can be evaluated as follows. Considering the element
$\sum_{n=0}^{\infty}p_{n}(x)\frac{y^{n}}{w_{n}}$ of $\mathbb{F}\ldbrack x,y \rdbrack$, $\alpha(y)$ is obtained by the substitution $x=0$. Moreover, if we consider the element of $\mathbb{F}\ldbrack x,y \rdbrack$ by which $T_{h,W}$ multiplies this series, then $\beta(y)$ is $w_{1}$ times the term of it which is linear in $h$.
\label{ThWser}
\end{cor}

\begin{proof}
We recall that $\alpha(y)$ is $C_{A_{0},W}(y)$, which is the constant term (in $x$) of our element of $\mathbb{F}\ldbrack x,y \rdbrack$ (in the second presentation). Moreover, the proof of Proposition \ref{Shefmul} shows that the multiplier in question is $W\big(h\beta(y)\big)$. The corollary easily follows from these assertions.
\end{proof}

\smallskip

We now describe the two natural subgroups of $R_{W}$ arising from its structure as a semi-direct product in terms of special Sheffer sequences.
\begin{defn}
A sequence $\{p_{n}(x)\}_{n\in\mathbb{N}}$ of polynomials is called an \emph{Appell sequence of weight $W$}, or a \emph{$W$-Appell sequence}, if its associated operator $Q_{A,W}$ is $D_{W}$. The sequence is said to be \emph{of $W$-binomial type} in case evaluating the image of $\frac{p_{n}}{w_{n}}$ under $T_{h,W}$ at $x$ yields $\sum_{k=0}^{n}\frac{p_{n-k}(h)p_{k}(x)}{w_{n-k}w_{k}}$. \label{Appbin}
\end{defn}

The connection of Definition \ref{Appbin} with the previous notions is given in
\begin{prop}
A sequence $\{p_{n}(x)\}_{n\in\mathbb{N}}$ is $W$-Appell if and only if it is $W$-Sheffer and the polynomials $d_{n,k,A,W}(h)$ are just $h^{n-k}$, i.e., its $\beta$-parameter is trivial. \label{Appbeta1}
\end{prop}

\begin{proof}
The argument proving Lemma \ref{Shefbin} with $p_{n}(x)=x^{n}$ and $Q_{A,W}=D_{W}$ shows that $D_{W}$ and $T_{h,W}$ commute. $W$-Appell sequences are thus $W$-Sheffer. The form of the multiplier from the proof of Proposition \ref{Shefmul} (or Corollary \ref{ThWser}) shows that the conditions $d_{n,k,A,W}(h)=h^{n-k}$ for all $n \geq k$, i.e., $d_{l}(h)=h^{l}$ for every $l$, and $\beta(y)=y$, are equivalent. Now, Proposition \ref{genfunc} and the $W$-Sheffer condition yield the equality
$\sum_{n=0}^{\infty}p_{n}(x)\frac{y^{n}}{w_{n}}=\alpha(y)\sum_{k=0}^{\infty}\frac{\beta(y)^{k}x^{k}}{w_{k}}$. Moreover, the usual linear independence argument shows that two operators on $\mathbb{F}[x]$ are the same if and only if the images of the above expression under them coincide. Now, the definitions of $D_{W}$ and $Q_{A,W}$ show that the latter operator multiplies the left hand side by $y$, while the former one multiplies the right hand side by $\beta(y)$. Hence if $Q_{A,W}=D_{W}$ then $\beta(y)=y$, and conversely if $\beta(y)=y$ then the linear independence argument implies that $Q_{A,W}=D_{W}$. This proves the proposition.
\end{proof}

This connection continues with
\begin{prop}
A sequence $\{p_{n}(x)\}_{n\in\mathbb{N}}$ is of $W$-binomial type if and only if it is $W$-Sheffer and satisfies $p_{n}(0)=\delta_{n,0}$, i.e., if its $\alpha$-parameter is trivial. \label{binalpha1}
\end{prop}
The symbol $\delta_{n,0}$ here is the Kronecker delta symbol, which equals 1 if $n=0$ and 0 otherwise.

\begin{proof}
The condition for $W$-binomiality can be rephrased as the assertion that $d_{n,k,A,W}=p_{n-k}$ for every $n \geq k$. Hence $W$-binomial sequences are $W$-Sheffer by Lemma \ref{Shefbin}. Corollary \ref{ThWser} shows the equivalence of $\alpha(y)=1$ and the condition $p_{n}(0)=\delta_{n,0}$ for all $n$. As $T_{0,W}$ is the identity map and the $p_{k}$ are linearly independent, substituting $h=0$ in Definition \ref{Appbin} shows that $W$-binomial sequences must satisfy the condition $p_{n}(0)=\delta_{n,0}$. Conversely, if $\{p_{n}(x)\}_{n\in\mathbb{N}}$ is a $W$-Sheffer sequence with a
trivial $\alpha$-parameter, then Proposition \ref{genfunc} shows that $\sum_{n=0}^{\infty}p_{n}(x)\frac{y^{n}}{w_{n}}$ equals
$W\big(x\beta(y)\big)$ for some $\beta\in\mathbb{F}\ldbrack y \rdbrack$ with
$v(\beta)=1$. On the other hand, The proof of Proposition \ref{Shefmul} expands the multiplier $W\big(h\beta(y)\big)$ as $\sum_{n=0}^{\infty}d_{n}(h)\frac{y^{n}}{w_{n}}$. This implies the the equality $d_{n}=p_{n}$ for every $n$, which completes the proof of the proposition.
\end{proof}

\smallskip

We have described several notions in terms of the behavior with respect to $T_{h,W}$. Many of them could have equivalently been defined using only $D_{W}$, as one sees from the following
\begin{prop}
Any operator on $\mathbb{F}[x]$ commutes with $D_{W}$ if and only if it commutes with all the $W$-translations $T_{h,W}$. \label{DWThW}
\end{prop}

\begin{proof}
The definitions of $T_{h,W}$ and $D_{W}$ show that $T_{h,W}\big(\frac{x^{n}}{w_{n}}\big)$ can be written as $\sum_{l=0}^{\infty}\frac{h^{l}}{w_{l}}D_{W}^{l}\big(\frac{x^{n}}{w_{n}}\big)$. Hence we can write $T_{h,W}=\sum_{l=0}^{\infty}\frac{h^{l}}{w_{l}}D_{W}^{l}=W(hD_{W})$ as operators on $\mathbb{F}[x]$ (this is well-defined on $\mathbb{F}[x]$ since any
polynomial is annihilated by a high enough power of $D_{W}$). Hence if an operator $Z$ on $\mathbb{F}[x]$ commutes with $D_{W}$ then it has to commute also with $T_{h,W}$ for any $h$. Conversely, if $ZT_{h,W}=T_{h,W}Z$ then expanding $T_{h,W}$ as above and comparing the terms which are linear in $h$ yields the desired equality $ZD_{W}=D_{W}Z$. This proves the proposition.
\end{proof}

The above arguments prove, via Proposition \ref{DWThW}, that the sequence of polynomials $\{p_{n}(x)\}_{n\in\mathbb{N}}$ corresponding to the matrix $A \in L$ is $W$-Sheffer if and only if the associated operator $Q_{A,W}$ commutes with $D_{W}$. Writing $D_{W}\big(\frac{p_{n}}{w_{n}}\big)(x)$ as $\sum_{k=0}^{n}\frac{\tilde{d}_{n,k,A,W}}{w_{n-k}}\cdot\frac{p_{k}(x)}{w_{k}}$ with $\tilde{d}_{n,k,A,W}\in\mathbb{F}$, we see that $\{p_{n}(x)\}_{n\in\mathbb{N}}$ is $W$-Sheffer if and only if $\tilde{d}_{n,k,A,W}$ depends only on $n-k$ (hence can be written as $\tilde{d}_{n-k}$). It is also equivalent to $D_{W}$ multiplying
$\sum_{n=0}^{\infty}p_{n}(x)\frac{y^{n}}{w_{n}}$ by a power series in $y$, which is then $\sum_{l=0}^{\infty}\tilde{d}_{l}\frac{y^{l}}{w_{l}}=\beta(y)$.

\section{Linear Functionals on Polynomials \label{LinFunc}}

We now consider the space $\mathbb{F}[x]^{*}$ of linear functionals on $\mathbb{F}[x]$. Fix a power series $W(t)=\sum_{n=0}^{\infty}\frac{t^{n}}{w_{n}}\in\mathbb{F}\ldbrack t \rdbrack$ as above.
\begin{lem}
The bilinear map taking two elements $\varphi$ and $\psi$ of $\mathbb{F}[x]^{*}$ to the linear functional defined by $\varphi\cdot_{W}\psi\big(\frac{x^{n}}{w_{n}}\big)=\sum_{k=0}^{n}\varphi\big(\frac{x^{k}}{w_{k}}\big)\psi\big(\frac{x^{n-k}}{w_{n-k}}\big)$ defines a product on $\mathbb{F}[x]^{*}$, making it a ring which is isomorphic to $\mathbb{F}\ldbrack y \rdbrack$. \label{Fx*ring}
\end{lem}

\begin{proof}
For any $n\in\mathbb{N}$ we denote $\varphi\big(\frac{x^{n}}{w_{n}}\big)$ by $b_{n}$ and $\psi\big(\frac{x^{n}}{w_{n}}\big)$ by $c_{n}$. We then identify $\varphi$ and $\psi$ with the elements $\sum_{n=0}^{\infty}b_{n}y^{n}$ and $\sum_{n=0}^{\infty}c_{n}y^{n}$ of $\mathbb{F}\ldbrack y \rdbrack$ respectively. The fact that the elements $\frac{x^{n}}{w_{n}}$ are a basis of $\mathbb{F}[x]$ implies that this identification is an isomorphism of vector spaces between
$\mathbb{F}[x]^{*}$ and $\mathbb{F}\ldbrack y \rdbrack$. As the product of the series $\sum_{n=0}^{\infty}b_{n}y^{n}$ and $\sum_{n=0}^{\infty}c_{n}y^{n}$ in $\mathbb{F}\ldbrack y \rdbrack$ equals $\sum_{n=0}^{\infty}\big(\sum_{k=0}^{n}b_{k}c_{n-k}\big)y^{n}$ and $\varphi\cdot_{W}\psi\big(\frac{x^{n}}{w_{n}}\big)=\sum_{k=0}^{n}b_{k}c_{n-k}$, this isomorphism transfers the product from $\mathbb{F}\ldbrack y \rdbrack$ to the operation $\cdot_{W}$ on $\mathbb{F}[x]^{*}$. Carrying the ring axioms from $\mathbb{F}\ldbrack y \rdbrack$ to $\mathbb{F}[x]^{*}$ via this isomorphism, this completes the proof of the lemma.
\end{proof}
A useful interpretation of Lemma \ref{Fx*ring} is the observation that $\varphi\in\mathbb{F}[x]^{*}$ is identified, in this isomorphism, with the element of $\mathbb{F}\ldbrack y \rdbrack$ to which it sends the expression $W(xy)=\sum_{n=0}^{\infty}\frac{x^{n}y^{n}}{w_{n}}$. This element of $\mathbb{F}\ldbrack x,y \rdbrack$ corresponds, via Proposition \ref{genfunc}, to the trivial element of $R_{W}$. Carrying the valuation from $\mathbb{F}\ldbrack y \rdbrack$ to $\mathbb{F}[x]^{*}$, with its multiplicative property, via the isomorphism from Lemma \ref{Fx*ring}, we get the following immediate
\begin{cor}
For $\varphi\in\mathbb{F}[x]^{*}$ define
$v(\varphi)=\min\big\{n\in\mathbb{N}\big|\varphi(x^{n})\neq0\big\}$ (and
$\varphi(0)=\infty$). Then $v$ is multiplicative. \label{vFx*}
\end{cor}
The number $v(\varphi)$ from Corollary \ref{vFx*} is also the minimal number $n$ such that $\varphi(p)=0$ wherever the degree of $p$ is strictly smaller than $n$. Hence this parameter is independent of $W$.

\smallskip

Let $\varepsilon_{h}\in\mathbb{F}[x]^{*}$ be the functional arising from evaluation at $h$: $\varepsilon_{h}(p)=p(h)$. Then $\varepsilon_{h}$ corresponds to the power series $W(hy)\in\mathbb{F}\ldbrack y \rdbrack$, and in particular $\varepsilon_{0}$ corresponds to 1 and is the identity of $\mathbb{F}[x]^{*}$. The relation between multiplication of linear functionals and operators on polynomials is given in the following
\begin{lem}
Let $S$ be a linear operator on $\mathbb{F}[x]$. Then the operation $\varphi\mapsto\varphi \circ S$ on $\mathbb{F}[x]^{*}$ is obtained through multiplication with some element $\psi\in\mathbb{F}[x]^{*}$ if and only if $S$ does not increase the degrees of polynomials and commutes with $D_{W}$. In this case $\psi$ is uniquely determined as $\varepsilon_{0} \circ S$. For the particular case $S=T_{h,W}$ we get $\psi=\varepsilon_{h}$. \label{opFx*}
\end{lem}

\begin{proof}
Write $S\big(\frac{x^{n}}{w_{n}}\big)$ as $\sum_{k=0}^{d_{k}}\frac{b_{n,k,S}}{w_{n-k}}\cdot\frac{x^{k}}{w_{k}}$ for some finite degree $d_{k}$, so that the composition gives $\varphi \circ S\big(\frac{x^{n}}{w_{n}}\big)=\sum_{k=0}^{d_{k}}\frac{b_{n,k,S}}{w_{n-k}}\varphi\big(\frac{x^{k}}{w_{k}}\big)$. Fixing some $\psi\in\mathbb{F}[x]^{*}$ and comparing this with $(\varphi\cdot_{W}\psi)\big(\frac{x^{n}}{w_{n}}\big)=\sum_{k=0}^{n}\psi\big(\frac{x^{n-k}}{w_{n-k}}\big)\varphi\big(\frac{x^{k}}{w_{k}}\big)$, we get an equality for every $\varphi\in\mathbb{F}[x]^{*}$ if and only if $\frac{b_{n,k,S}}{w_{n-k}}=\psi\big(\frac{x^{n-k}}{w_{n-k}}\big)$ for every $n$ and $k$. But this is equivalent to $b_{n,k,S}$ depending only on $n-k$ and vanishing for $n<k$ (so that the degree of $S(p)$ cannot exceed that of $p$), together with the equality $\psi(x^{l})=b_{l,0,S}$. The first assertion now follows as in the proof of Lemma \ref{Shefbin}, if we replace $Q_{A,W}$ by $D_{W}$, $T_{h,W}$ by $S$, and $d_{n,k,A,W}(h)$ by $b_{n,k,S}$. For evaluating $\psi$ we either observe that both $\varepsilon_{0} \circ S\big(\frac{x^{n}}{w_{n}}\big)$ and $\psi\big(\frac{x^{n}}{w_{n}}\big)$ are equal $\frac{b_{n,0,S}}{w_{n}}$ for every $n\in\mathbb{N}$, or deduce the equality $\psi=\varepsilon_{0}\cdot_{W}\psi=\varepsilon_{0} \circ S$ from the fact that $\varepsilon_{0}$ is the multiplicative identity of $\mathbb{F}[x]^{*}$. For $S=T_{h,W}$ the coefficient $b_{n,k,S}$ equals $h^{n-k}$ for any $n$ and $k$ (by definition), so that the equality $\varphi \circ T_{h,W}=\varphi\cdot_{W}\varepsilon_{h}$ follows from the calculations above. This completes the proof of the lemma.
\end{proof}

From this we may deduce another characterization of Sheffer sequences, as is
given in
\begin{thm}
A graded sequence $\{p_{n}(x)\}_{n\in\mathbb{N}}$ of polynomials is $W$-Sheffer if and only if there exists another graded sequence $\{d_{l}(x)\}_{l\in\mathbb{N}}$ of polynomials such that the equality $\varphi\cdot_{W}\psi\big(\frac{p_{n}}{w_{n}}\big)=\sum_{k=0}^{n}\varphi\big(\frac{p_{k}}{w_{k}}\big)\psi\big(\frac{d_{n-k}}{w_{n-k}}\big)$ holds for every $\varphi$ and $\psi$ from $\mathbb{F}[x]^{*}$. \label{ShefFx*}
\end{thm}

\begin{proof}
If such a sequence $\{d_{l}(x)\}_{l\in\mathbb{N}}$ exists, then let $\psi$ be the unique element of $\mathbb{F}[x]^{*}$ which sends $\frac{d_{1}}{w_{1}}$ to 1 and annihilates all the other $d_{l}$s. With this $\psi$ we get the equality $\varphi\cdot_{W}\psi\big(\frac{p_{n}}{w_{n}}\big)=\varphi\big(\frac{p_{n-1}}{w_{n-1}}\big)=(\varphi \circ Q_{A,W})\big(\frac{p_{n}}{w_{n}}\big)$ for every $\varphi\in\mathbb{F}[x]^{*}$ and $n\in\mathbb{N}$. But then $Q_{A,W}$ must commute with $D_{W}$ by Lemma \ref{opFx*}, so that $\{p_{n}(x)\}_{n\in\mathbb{N}}$ is $W$-Sheffer by Proposition \ref{DWThW}. Conversely, applying $\varphi$ to the expression from Proposition
\ref{genfunc} yields the equality
\[\sum_{n=0}^{\infty}\varphi\bigg(\frac{p_{n}}{w_{n}}\bigg)(x)y^{n}=\varphi\Big(\alpha(y)W\big(x\beta(y)\big)\Big)=
\alpha(y)\sum_{l=0}^{\infty}\beta^{l}(y)\varphi\bigg(\frac{x^{l}}{w_{l}}\bigg).\] Replace $\varphi$ by a product $\varphi\cdot_{W}\psi$, and decompose
$\varphi\cdot_{W}\psi\big(\frac{x^{l}}{w_{l}}\big)$ according to the definition. The right hand side is now easily seen to be the same one but multiplied by $\sum_{r=0}^{\infty}\beta^{r}(y)\psi\big(\frac{x^{r}}{w_{r}}\big)$. But the same argument allows us to write this multiplier as $\sum_{n=0}^{\infty}\psi\big(\frac{d_{n}}{w_{n}}\big)(x)y^{n}$, where $\{d_{l}(x)\}_{l\in\mathbb{N}}$ is the $W$-binomial sequence with the same $\beta$-parameter as $\{p_{n}(x)\}_{n\in\mathbb{N}}$. Comparing the coefficient of $y^{n}$ in this product with the coefficient $\varphi\cdot_{W}\psi\big(\frac{p_{n}}{w_{n}}\big)(x)$ from the left hand side now yields the desired equality. This completes the proof of the theorem.
\end{proof}

\smallskip

We turn to a yet another description of Sheffer sequences, in terms of dual bases. If a sequence $\{\varphi_{r}\}_{r\in\mathbb{N}}$ of functionals in $\mathbb{F}[x]^{*}$ satisfies $v(\varphi_{r})\to\infty$ as $r\to\infty$, and $\{c_{r}\}_{r\in\mathbb{N}}$ is any sequence of scalars, then the sum $\sum_{r=0}^{\infty}c_{r}\varphi_{r}$ produces a well-defined element of $\mathbb{F}[x]^{*}$. Indeed, $\sum_{r=0}^{\infty}c_{r}\varphi_{r}(p)$ is finite for every polynomial $p\in\mathbb{F}[x]$ (this also follows by transferring to $\mathbb{F}\ldbrack y \rdbrack$). Conversely, the convergence of, e.g., $\sum_{r=0}^{\infty}\varphi_{r}$ in this sense implies the condition on $v(\varphi_{r})$. In this case we shall call $\{\varphi_{r}\}_{r\in\mathbb{N}}$ a \emph{pseudo-basis} of $\mathbb{F}[x]^{*}$ if any $\eta\in\mathbb{F}[x]^{*}$ can be presented as $\sum_{r=0}^{\infty}c_{r}\varphi_{r}$ with a unique sequence of scalars $\{c_{r}\}_{r\in\mathbb{N}}$. We shall need
\begin{lem}
Let $\{\varphi_{r}\}_{r\in\mathbb{N}}\subseteq\mathbb{F}[x]^{*}$ be a pseudo-basis of $\mathbb{F}[x]^{*}$, and define $b_{r,l}$ to be the uniquely determined coefficient such that $\sum_{r=0}^{\infty}b_{r,l}\varphi_{r}$ is the element of $\mathbb{F}[x]^{*}$ which takes each monomial $\frac{x^{k}}{w_{k}}$ to $\delta_{k,l}$. Fixing $r\in\mathbb{N}$, we have $b_{r,l}=0$ for any larger enough $l$. \label{rowfin}
\end{lem}

\begin{proof}
For any $d$ and $s$ in $\mathbb{N}$ we denote $V_{s,d}$ the space of all finite sequences $\{c_{r}\}_{r=0}^{s}$ that admit an extension to an infinite sequence $\{c_{r}\}_{r=0}^{\infty}$ such that $v\big(\sum_{r=0}^{\infty}c_{r}\varphi_{r}\big) \geq d$. It is clear that $V_{s,d+1} \subseteq V_{s,d}$, and we define $U_{s}=\bigcap_{d=0}^{\infty}V_{s,d}$. It is clear from finite dimensionality that $V_{s,d}=U_{s}$ for large enough $d$.  Since we consider the existence of extensions, the natural restriction map of sequences takes $V_{s,d+1}$ surjectively onto $V_{s,d}$. It thus follows that $U_{s+1}$ maps surjectively onto $U_{s}$.

We claim that $U_{s}=\{0\}$ for all $s$. Indeed, assume that $U_{s}\neq\{0\}$ for some $s$. Starting from a non-zero element of $U_{s}$, we choose a pre-image in $U_{s+1}$, and continuing in this manner we get a non-zero sequence $\{c_{r}\}_{r=0}^{\infty}$ whose finite beginnings up to $s$ lies in $U_{s}$ for any $s\in\mathbb{N}$. We claim that $\sum_{r=0}^{\infty}c_{r}\varphi_{r}=0$. Indeed, given any $d\in\mathbb{N}$ we choose $s$ such that $v(\varphi_{r})>d$ for all $r>s$, and then $\{c_{r}\}_{r=0}^{s} \in U_{s} \subseteq V_{s,d+1}$ admits some continuation such that $v\big(\sum_{r=0}^{s}c_{r}\varphi_{r}+\sum_{r=s+1}^{\infty}\tilde{c}_{r}\varphi_{r}\big)>d$. But the difference between this functional and the one under consideration is based only on $\varphi_{r}$ with valuation larger than $d$. Hence $v\big(\sum_{r=0}^{\infty}c_{r}\varphi_{r}\big)>d$ for every $d$, which proves our claim. But the equality $\sum_{r=0}^{\infty}c_{r}\varphi_{r}=0=\sum_{r=0}^{\infty}0\varphi_{r}$, where the $c_{r}$s are not all 0s, contradicts the assumption that $\{\varphi_{r}\}_{r\in\mathbb{N}}$ is a pseudo-basis. This proves that all the spaces $U_{s}$ are trivial.

Now fix $s\in\mathbb{N}$. As $U_{s}=\{0\}$, we have $V_{s,d}=0$ for large enough $d$. As $v\big(\sum_{r=0}^{\infty}b_{r,l}\varphi_{r}\big)=l$ by definition, we find that $\{b_{r,l}\}_{r=0}^{s}$ is an element of $V_{s,d}$ for every $l \geq d$. Hence this sequence must vanish for large enough $l$, which yields the desired assertion for every $0 \leq r \leq s$. This completes the proof of the lemma.
\end{proof}

Let now $\{p_{n}(x)\}_{n\in\mathbb{N}}$ be an arbitrary (not necessarily graded) sequence of polynomials, which we assume to be a basis for $\mathbb{F}[x]$. We say that a sequence of functionals $\{\varphi_{r}\}_{r\in\mathbb{N}}\subseteq\mathbb{F}[x]^{*}$ is the \emph{$W$-dual basis} of $\{p_{n}(x)\}_{n\in\mathbb{N}}$ if the equality $\varphi_{r}\big(\frac{p_{n}}{w_{n}}\big)=\delta_{n,r}$ holds for every $n$ and $r$. We now have
\begin{prop}
Any $W$-dual basis is a pseudo-basis. Conversely, any pseudo-basis $\{\varphi_{r}\}_{r\in\mathbb{N}}$ of $\mathbb{F}[x]^{*}$ is $W$-dual to a unique basis $\{p_{n}(x)\}_{n\in\mathbb{N}}$ of $\mathbb{F}[x]$. In addition, this sequence of polynomials is graded if and only if $v(\varphi_{r})=r$ for all $r$, and any sequence $\{\varphi_{r}\}_{r\in\mathbb{N}}$ with $v(\varphi_{r})=r$ arises in this way. \label{dualbas}
\end{prop}

\begin{proof}
Let $\{\varphi_{r}\}_{r\in\mathbb{N}}$ be the sequence which is $W$-dual to $\{p_{n}(x)\}_{n\in\mathbb{N}}$. We need to show that $v(\varphi_{r})\to\infty$ as $r\to\infty$. Fix $d\in\mathbb{N}$. As $\{p_{n}(x)\}_{n\in\mathbb{N}}$ is a basis, the monomials up to $x^{d}$ are linear combinations of the $p_{n}$s. Hence all of them involve only finitely many of the latter. If $p_{l}$ is the maximal one appearing in any of these monomials, then it is clear from the
definition (and from linearity) that $\varphi_{r}(x^{n})$ vanishes for every
$r>l$ and $n \leq d$. Hence $v(\varphi_{r})>d$ for any such $r$, establishing
the valuation condition. The $W$-duality now shows that for any sequence $\{a_{n}\}_{n\in\mathbb{N}}$ of scalars, there exists a unique combination, namely $\sum_{r=0}^{\infty}\frac{a_{r}}{w_{r}}\varphi_{r}$, which takes each $p_{n}$ to $a_{n}$. In particular any $\psi\in\mathbb{F}[x]^{*}$ is represented uniquely as $\sum_{r=0}^{\infty}\psi\big(\frac{p_{r}}{w_{r}}\big)\varphi_{r}$. This proves the first assertion.

Conversely, let $\{\varphi_{r}\}_{r\in\mathbb{N}}$ be a pseudo-basis, and denote $\varphi_{r}\big(\frac{x^{k}}{w_{k}}\big)$ by $s_{k,r}$. We have $s_{k,r}=0$ if $v(\varphi_{r})>k$, which happens, for fixed $k$, if $r$ is large enough. The matrix $S$ with entries $s_{k,r}$ is therefore row-finite. Let $b_{r,l}$ be the coefficients from Lemma \ref{rowfin} ($\{\varphi_{r}\}_{r\in\mathbb{N}}$ is a pseudo-basis). This translates to the equality $\sum_{r=0}^{\infty}s_{k,r}b_{r,l}=\delta_{k,l}$, an essentially finite equality holding for every $k$ and $l$. As Lemma \ref{rowfin} shows that
the matrix $B$ with entries $b_{r,l}$ is also row-finite, the expression $p_{n}(x)=w_{n}\sum_{k=0}^{\infty}b_{n,k}\frac{x^{k}}{w_{k}}$ is a polynomial  for each $n$. It follows that $\varphi_{r}\big(\frac{p_{n}}{w_{n}}\big)$ equals $\sum_{k=0}^{\infty}b_{n,k}s_{k,r}$ (a finite sum), the $(n,r)$ entry of $BS$, for every $n$ and $r$. We now consider the row-finite matrices $B$ and $S$ as representing linear operators on a vector space with a countable basis. We have seen that the product $SB$ is the identity. But the pseudo-basis property shows that $S$ is invertible, so that $B=S^{-1}$ and $BS$ is the identity as well. Hence $\varphi_{r}\big(\frac{p_{n}}{w_{n}}\big)=\delta_{n,r}$ for every $n$ and $r$. The $p_{n}$s are therefore linearly independent, and from the relation $S=B^{-1}$ we deduce the equality $\frac{x_{l}}{w_{l}}=\sum_{n=0}^{\infty}s_{l,n}\frac{p_{n}}{w_{n}}$ for every $l\in\mathbb{N}$. It follows that the $p_{n}$s form a basis for $\mathbb{F}[x]$, and $\{\varphi_{r}\}_{r}$ is the $W$-dual basis. This proves the second assertion.

For the third assertion, note that if $\{p_{n}(x)\}_{n\in\mathbb{N}}$ is graded then $\{p_{n}(x)\}_{n=0}^{l}$ span the space of polynomials of degree at most $l$. Hence $\varphi_{r}$ vanishes on each polynomial of degree smaller than $r$, but not on $p_{r}$ of degree $r$, showing that $v(\varphi_{r})=r$. Conversely, if $v(\varphi_{r})=r$ then the matrix $S$ from the previous paragraph lies in the group $L$ from Proposition \ref{lowtrigrad}. Hence its inverse $B$ lies in $L$ as well, showing that the corresponding sequence $\{p_{n}(x)\}_{n\in\mathbb{N}}$ is graded (and in particular forms a basis for $\mathbb{F}[x]$). This completes the proof of the proposition.
\end{proof}

Fix elements $\xi$ and $\eta$ of $\mathbb{F}[x]^{*}$ with $v(\xi)=0$ and $v(\eta)=1$. Corollary \ref{vFx*} shows that wherever $v(\xi)=0$ and $v(\eta)=1$ the sequence in which $\varphi_{r}=\xi\cdot_{W}\eta_{W}^{r}$ (where $\eta_{W}^{r}$ stands for the $r$th power of $\eta$ in the product $\cdot_{W}$) satisfies $v(\varphi_{r})=r$ for any $r\in\mathbb{N}$. It is therefore a pseudo-basis which is $W$-dual to some graded sequence of polynomials, by Proposition \ref{dualbas}. We now prove
\begin{thm}
The graded sequence $\{p_{n}(x)\}_{n\in\mathbb{N}}$ is $W$-Sheffer if and only
if its $W$-dual basis takes the form $\varphi_{r}=\xi\cdot_{W}\eta_{W}^{r}$ for
some elements $\xi$ and $\eta$ of $\mathbb{F}[x]^{*}$ with $v(\xi)=0$ and
$v(\eta)=1$. \label{Shefdual}
\end{thm}

\begin{proof}
Let $\{p_{n}(x)\}_{n\in\mathbb{N}}$ be a graded sequence in which $p_{n}(x)=\sum_{k=0}^{n}a_{n,k}x^{k}$, and define $A$ to be the associated element of $L$ from Proposition \ref{seqmat} as well as $S$, with the entries $s_{k,r}=\varphi_{r}\big(\frac{x^{k}}{w_{k}}\big)$, as in the proof of Proposition \ref{dualbas}. This proof shows that $S$ is the matrix which is inverse to $B$ in which $b_{n,k}=\frac{a_{n,k}w_{k}}{w_{n}}$. Hence the $(k,r)$-entry of $A^{-1}$ is $\frac{w_{k}s_{k,r}}{w_{r}}$. The sum $\sum_{k=0}^{\infty}s_{k,r}y^{k}$ therefore forms, on the one hand, the series $C_{A^{-1}_{r},W}(y)$ associated to the element $A^{-1} \in L$, multiplied by $w_{r}$. On the other hand, it is the image of $\sum_{k=0}^{\infty}\frac{x^{k}y^{k}}{w_{k}}=W(xy)$ under $\varphi_{r}$, i.e., the element of $\mathbb{F}\ldbrack y \rdbrack$ which is associated to $\varphi_{r}$ in Lemma \ref{Fx*ring}. Using the multiplicative structure from that Lemma, we find that $\varphi_{r}$ is of the form $\xi\cdot_{W}\eta_{W}^{r}$ for some $\xi$ and $\eta$ as above if and only if the columns of $A^{-1}$ satisfy the condition of columns of elements of $R_{W}$ from Definition \ref{Riordan}. But this is equivalent to $A$ being in $R_{W}$ via Proposition \ref{Rgroup}, hence to $\{p_{n}(x)\}_{n\in\mathbb{N}}$ being $W$-Sheffer by Theorem \ref{comp}. This proves the theorem.
\end{proof}

To find the functionals $\xi$ and $\eta$ which generate the $W$-dual basis of a $W$-Sheffer sequence $\{p_{n}(x)\}_{n\in\mathbb{N}}$ as in Theorem \ref{Shefdual}, we consider the sequence of polynomials $\{d_{l}(x)\}_{l\in\mathbb{N}}$ from Theorem \ref{ShefFx*}. The proof of that theorem introduced the functional sending $\frac{d_{1}}{w_{1}}$ to 1 and the other $d_{l}$s to 0, which is precisely our $\eta$. The functional $\xi$ is the one taking $p_{0}$ to 1 and the other $p_{n}$s to 0. To see this, extend the proof of Theorem \ref{ShefFx*} by a simple induction to show that $\xi\cdot_{W}\eta_{W}^{r}$ sends $\frac{p_{n}}{w_{n}}$ to 0 if $n<r=v(\xi\cdot_{W}\eta_{W}^{r})$ and to $\xi\big(\frac{p_{n-r}}{w_{n-r}}\big)=\delta_{n,r}$ otherwise, showing that it is indeed the required $W$-dual basis. The proof of Theorem
\ref{ShefFx*} also characterizes $\xi$ and $\eta$ as those functionals which  take $W(xy)$ to the $\alpha$ and $\beta$ parameters of $A^{-1}$ respectively.  Another description of the $W$-dual basis $\{\varphi_{r}\}_{r\in\mathbb{N}}$ of the sequence $\{p_{n}(x)\}_{n\in\mathbb{N}}$ is via the equality
$\varphi_{r}\big(\sum_{n=0}^{\infty}p_{n}(x)\frac{y^{n}}{w_{n}}\big)=y^{r}$, which for a $W$-Sheffer sequence becomes
$\varphi_{r}\big[\alpha(y)W\big(x\beta(y)\big)\big]=y^{r}$ via Proposition
\ref{genfunc}. These equalities may be used to provide alternative proofs of some of the results of this section.

\section{Group Operations and Conjugate Subgroups \label{GroupOp}}

We can consider the operators $D_{W}$, $T_{h,W}$, and $Q_{A,W}$ as infinite lower triangular matrices as well. Indeed, the formula defining the action of $D_{W}$ shows that it coincides with the operation of the matrix, which we denote $M_{W}$, whose $(n,k)$-entry is $\frac{w_{n}}{w_{n-1}}$ if $k=n-1$ and 0 otherwise. This allows us to give another characterization of the $W$-Appell sequences, as in
\begin{prop}
The sequence $\{p_{n}(x)\}_{n\in\mathbb{N}}$, corresponding to $A \in L$, is $W$-Appell if and only if the matrix $A$ can be obtained by substituting $y=M_{W}$ inside a power series from $\mathbb{F}\ldbrack y \rdbrack$ with valuation 0. Explicitly, we have $A=\alpha(M_{W})$ with the $\alpha$-parameter of $A \in R_{W}$. \label{AppserDW}
\end{prop}
Note that as every power of $M_{W}$ is supported on a different diagonal, we can substitute $M_{W}$ inside power series, and there is no problem of convergence.

\begin{proof}
Evaluating the powers of $M_{W}$ shows that $(M_{W}^{r})_{n,k}=\frac{w_{n}}{w_{n-r}}\delta_{k,n-r}$ for any $r\in\mathbb{N}$. Hence taking a combination of the sort $\sum_{r=0}^{\infty}\frac{b_{r}}{w_{r}}M_{W}^{r}$ yields a matrix whose $(n,k)$-entry is $\frac{b_{n-k}w_{n}}{w_{k}w_{n-k}}$ (and 0 if $n<k$). This matrix is in $L$ if and only if $b_{0}\neq0$, i.e., the power series $\sum_{r=0}^{\infty}\frac{b_{r}}{w_{r}}y^{r}$ has valuation 0 in $\mathbb{F}\ldbrack y \rdbrack$. On the other hand, Proposition \ref{Appbeta1} shows that $\{p_{n}(x)\}_{n\in\mathbb{N}}$ is $W$-Appell if and only if $\sum_{n=0}^{\infty}p_{n}(x)\frac{y^{n}}{w_{n}}=\alpha(y)W(xy)$ for some $\alpha\in\mathbb{F}\ldbrack y \rdbrack$ with $v(\alpha)=0$. The left hand side is $\sum_{n=0}^{\infty}\sum_{k=0}^{n}a_{n,k}\frac{x^{k}y^{n}}{w_{n}}$, while the right hand side expands as $\sum_{l=0}^{\infty}\frac{a_{l,0}y^{l}}{w_{l}}\sum_{k=0}^{\infty}\frac{x^{k}y^{k}}{w_{k}}$ (since $\alpha=C_{A_{0},W}$). From this we deduce, by comparing the coefficients of $x^{k}y^{n}$, that a necessary and sufficient condition for $\{p_{n}(x)\}_{n\in\mathbb{N}}$ to be $W$-Appell is that the equality $a_{n,k}=\frac{a_{n-k,0}w_{n}}{w_{k}w_{n-k}}$ holds for every $n$ and $k$. But this is equivalent to $A$ being equal to the matrix $\sum_{l=0}^{\infty}\frac{a_{l,0}}{w_{l}}M_{W}^{l}$. As $\alpha(y)$ is $C_{A_{0},W}(y)=\sum_{l=0}^{\infty}a_{l,0}\frac{y^{l}}{w_{l}}$, this completes the proof of the proposition.
\end{proof}

From Proposition \ref{AppserDW} we directly deduce
\begin{cor}
The operator $T_{h,W}$ corresponds to the matrix in $R_{W}$ with $\alpha$-parameter $\alpha(y)=W(hy)$ and trivial $\beta$-parameter. \label{ThWApp}
\end{cor}

\begin{proof}
The proof of Proposition \ref{DWThW} shows that $T_{h,W}=W(hD_{W})$ as operators on $\mathbb{F}[x]$. Hence it is represented by the matrix which is the result of substituting $y=M_{W}$ in the series $W(hy)$. The corollary now follows from the second assertion of Proposition \ref{AppserDW}.
\end{proof}

\smallskip

We now show that many the previous notions can be described in terms of a certain group action of $L$. We have seen that the operators $D_{W}$ can be represented by matrices. The same can be done with $Q_{A,W}$. In fact, the main idea for this presentation is given in the following
\begin{lem}
The operator $Q_{A,W}$ is described by the matrix $A^{-1}M_{W}A$. \label{action}
\end{lem}

\begin{proof}
Let $\widetilde{Q}$ be the matrix representing the operation of $Q_{A,W}$ on the powers of $x$. Then the sequence $\big\{Q_{A,W}(x^{n})\big\}_{n\in\mathbb{N}}$ is $\widetilde{Q}$ times the monomial sequence, with $m_{n}(x)=x^{n}$. The sequence $\big\{Q_{A,W}(p_{n})(x)\big\}_{n\in\mathbb{N}}$ can therefore be written as $A\widetilde{Q}$ times the monomial sequence, and multiplying from the left by the usual row vector yields the image of $\sum_{n=0}^{\infty}p_{n}(x)\frac{y^{n}}{w_{n}}$ under $Q_{A,W}$. But the proof of Proposition \ref{Appbeta1} shows that $Q_{A,W}$ multiplies that power series by $y$. On the other hand, multiplying our row vector by $M_{W}$ from the right also has the effect of multiplying it by the scalar $y$. As $\{p_{n}\}_{n\in\mathbb{N}}$ is $A$ times the monomial sequence (Proposition \ref{seqmat}), we find that the two matrices $A\widetilde{Q}$ and $M_{W}A$ have the property that putting them between our row vector and the monomial sequence yields the same expression $y\sum_{n=0}^{\infty}p_{n}(x)\frac{y^{n}}{w_{n}}$. The usual linear independence implies thus implies the equality of these two matrices, which completes the proof of the lemma.
\end{proof}

To make the proof of Lemma \ref{action} a bit more visible, we note that the explicit meaning of the equality $Q_{A,W}(p_{n})=\frac{w_{n}p_{n-1}}{w_{n-1}}$ is
\[\sum_{l=0}^{n}a_{n,l}Q_{A,W}(x^{l})=Q_{A,W}(p_{n})(x)=\frac{w_{n}p_{n-1}}{w_{n-1}}=\sum_{k=0}^{n-1}\frac{w_{n}a_{n-1,k}x^{k}}{w_{n-1}}.\] We write the coefficient $\frac{w_{n}a_{n-1,k}}{w_{n-1}}$ of $x_{k}$ as $\sum_{r=0}^{\infty}(M_{W})_{n,r}a_{r,k}$, and observe that the coefficient of $x^{k}$ is now the $(n,k)$-entry of $M_{W}A$ on the right hand side and the $(n,k)$-entry of $A\widetilde{Q}$ on the left hand side. This proves the desired equality.

In view of Lemma \ref{action} we consider the right action of $L$ on the set of strictly lower triangular matrices (i.e., those lower triangular matrices in which all the entries on the main diagonal also vanish), in which $A \in L$ takes a strictly lower triangular matrix $M$ to $A^{-1}MA$. Using this operation we obtain an alternative description of the groups of $W$-Appell and $W$-Sheffer sequences, as in
\begin{thm}
The group of $W$-Appell sequences is the stabilizer of $M_{W}$ in $L$. The group $R_{W}$ of $W$-Riordan arrays (or $W$-Sheffer sequences) is the normalizer of the latter group in $L$. \label{actdesc}
\end{thm}

\begin{proof}
The first assertion follows directly from Definition \ref{Appbin} and Lemma \ref{action}. For the second assertion, Corollary \ref{ThWApp} implies that the operation of $T_{h,W}$ on the monomial sequence is via the matrix $W(hM_{W})$. The commutation of the operators $T_{h,W}$ and $Q_{A,W}$ from Definition \ref{Sheffer} translates, via the latter assertion and Lemma \ref{action}, to the commutation of $A^{-1}M_{W}A$ and $W(hM_{W})$. This is equivalent to the assertion that $AW(hM_{W})A^{-1}$ and $M_{W}$ commute. Comparing the powers of $h$, we find that the latter condition holds if and only if $M_{W}$ commutes with $AM_{W}^{l}A^{-1}$ for every $l\in\mathbb{N}$ (alternatively, $Q_{A,W}$ commutes with $D_{W}$ by Proposition \ref{DWThW}, so that $M_{W}$ commutes with $AM_{W}A^{-1}$ and hence with all its powers). It follows, via Proposition \ref{AppserDW}, that $A \in R_{W}$ if and only if $M_{W}$ commutes with any matrix of the sort $ABA^{-1}$ where $B \in L$ corresponds to a $W$-Appell sequence. But Lemma \ref{action} and the first assertion here translate the latter condition, via Proposition \ref{seqmat}, to the statement that conjugation by $A$ takes $W$-Appell sequences to $W$-Appell sequences. The proof of the theorem is now complete.
\end{proof}

\smallskip

Theorem \ref{actdesc} gives us a way to relate the Riordan groups of different weights, as in
\begin{cor}
Let $W(t)=\sum_{n=0}^{\infty}\frac{t^{n}}{w_{n}}$ and $\widetilde{W}(t)=\sum_{n=0}^{\infty}\frac{t^{n}}{\widetilde{w}_{n}}$ be two power series in $\mathbb{F}\ldbrack t \rdbrack$, with $w_{0}=\widetilde{w}_{0}=1$ and $w_{n}\widetilde{w}_{n}\neq0$ for each $n\in\mathbb{N}$. Then the groups $R_{W}$ and $R_{\widetilde{W}}$ are conjugate in $L$, and this conjugation takes the subgroup of $W$-Appell sequences to the subgroup of $\widetilde{W}$-Appell sequences. \label{conjW}
\end{cor}

\begin{proof}
Let $U$ be the diagonal element of $L$ whose $n$th diagonal entry is $\frac{w_{n}}{\widetilde{w}_{n}}$. Then a simple calculation shows that $U^{-1}M_{W}U=M_{\widetilde{W}}$, i.e., the action of $U$ from Lemma \ref{action} takes $M_{W}$ to $M_{\widetilde{W}}$. But this implies that conjugation by $U$ takes the stabilizer of $M_{W}$ to the stabilizer of $M_{\widetilde{W}}$. It now follows that conjugation by $U$ takes the normalizer of the former stabilizer to normalizer of the latter stabilizer. The corollary hence follows from Theorem \ref{actdesc}.
\end{proof}

In fact, one may obtain a more precise assertion than Corollary \ref{conjW}, as one sees in the following
\begin{prop}
The conjugation by the matrix $U$ from the proof of Corollary \ref{conjW} takes the element $A \in R_{W}$ with parameters $\alpha$ and $\beta$ to the element of $R_{\widetilde{W}}$ having the same parameters. \label{conjpres}
\end{prop}
Recall that as the action from Lemma \ref{action} is from the right, the conjugation in Proposition \ref{conjpres} takes $A$ to $U^{-1}AU$.

\begin{proof}
We recall from Propositions \ref{seqpowser} and \ref{genfunc} that multiplying an element $A \in R_{W}$ by the row vector based on $W$ yields the row vector whose $k$th entry is $\alpha(y)\frac{\beta(y)^{k}}{w_{k}}$. Now, multiplying the row vector based on $\widetilde{W}$ by $U^{-1}$ yields the one which is based on $W$. Hence multiplying the former vector by $U^{-1}A$ produces the vector with the entries $\alpha(y)\frac{\beta(y)^{k}}{w_{k}}$. But multiplying the latter vector by $U$ turns the entries to $\alpha(y)\frac{\beta(y)^{k}}{\widetilde{w}_{k}}$. Applying Propositions \ref{seqpowser} and \ref{genfunc} once more proves the desired assertion.
\end{proof}

It follows from Proposition \ref{conjpres} that the subgroup of $W$-binomial sequences is also preserved under this conjugation (see Proposition \ref{binalpha1}, as well as Proposition \ref{Appbeta1} for re-proving this for Appell sequences).

\smallskip

The matrices $M_{W}$, for varying $W$, contain only one diagonal of non-zero entries. However, the operation of $L$ produces much more matrices. Hence we define a \emph{degree decreasing operator} to be any linear operator $Q$ on $\mathbb{F}[x]$ such that the degree of $Q(p)$ is one less than the degree of $p$ (and $Q(a)=0$ for $a\in\mathbb{F}$). The same argument proving Proposition \ref{seqmat} shows that these operators are in one-to-one correspondence with strictly lower triangular matrices $M$ none of whose entries with indices $(n,n-1)$ for $n\in\mathbb{N}$ vanish. It is clear that the action of $L$ preserves this set. In addition, Corollary \ref{conjW} and Proposition \ref{conjpres} extend to
\begin{prop}
The action of $L$ on the set of degree decreasing operators is transitive. Moreover, the set of elements $A \in L$ that satisfy $a_{n,0}=\delta_{n,0}$ for every $n$ is a subgroup of $L$, with respect to which the set of degree decreasing operators is a principal homogenous space. \label{acttrans}
\end{prop}
We recall that a \emph{principal homogenous space} for a group $G$ is a set on which $G$ operates transitively with trivial stabilizers.

\begin{proof}
The fact that this subset of $L$ is a subgroup is easily verified by matrix multiplication. We fix one degree decreasing operator, the one corresponding to the matrix $M_{W}$ for $W(t)=\frac{1}{1-t}=\sum_{n=0}^{\infty}t^{n}$ (i.e., with $w_{n}=1$ for all $n$). It suffices to show that for any strictly lower triangular matrix $M$ representing a degree decreasing operator there exists a unique element $A \in L$ such that $AM_{W}A^{-1}=M$ and $a_{n,0}=\delta_{n,0}$. Denote the entries of $M$ by $m_{n,k}$, and compare the entries of $AM_{W}$ and $MA$. The $(n,k)$-entry of the former product is just $a_{n,k+1}$, while the same entry of the latter one is $\sum_{l=k}^{n-1}m_{n,l}a_{l,k}$. Hence $AM_{W}=MA$ if and only if every column of $A$ can be obtained from the previous one through multiplication by $M$. Therefore such $A$ always exists, and simple induction shows that any choice of 0th column gives a lower triangular matrix, which lies in $L$ if $a_{0,0}\neq0$. Fixing this column according to our hypothesis on $A$ thus completes the  proof of the proposition.
\end{proof}

\smallskip

Recall that the product on $\mathbb{F}[x]^{*}$ appearing in Lemma \ref{Fx*ring} is also based on $W$. We believe that there should be an action of $L$ on the set of multiplications on $\mathbb{F}[x]^{*}$ making it isomorphic to $\mathbb{F}\ldbrack y \rdbrack$ and preserving valuations, such that the stabilizer of the multiplication $\cdot_{W}$ is again the subgroup of $L$ corresponding to $W$-Appell sequences. However, as such multiplications are, in some sense, 3-dimensional objects (i.e., are represented by algebraic objects whose entries have 3 indices), we do not pursue this subject here further. In addition, recalling that a polynomial sequence being $W$-Sheffer implies many combinatorial properties of this sequence, we conclude this section by suggesting that higher Sheffer sequences, defined by taking the normalizer in $L$ of $R_{W}$ and repeating this construction, may also turn out to have some combinatorial importance as well. It is reasonable to conjecture that every such iteration would increase the type, in the terminology of \cite{[Sh]}, by 1. We leave this question, however, for further research.

\section{Some Examples and Relations \label{ExRel}}

We present the most classical and natural examples for the choice of $W$, with the resulting operators.

\noindent\textbf{Example: Exponentials}. Assume the $\mathbb{F}$ is of characteristic 0, fix $0\neq\lambda\in\mathbb{F}$, and take $w_{n}=\lambda^{n}n!$. In this case we have $W(t)=e^{t/\lambda}$, and the operator $D_{W}$, taking $\frac{x^{n}}{\lambda^{n}n!}$ to $\frac{x^{n-1}}{\lambda^{n-1}(n-1)!}$, is the usual derivative $\frac{d}{dx}$ multiplied by $\lambda$. The $W$-translation $T_{h,W}$ takes $x^{n}$ to $\sum_{k=0}^{n}\binom{n}{k}h^{n-k}x^{k}=(x+h)^{n}$ (in correspondence with the description of $T_{h,W}$ as $e^{hD_{W}/\lambda}=e^{hd/dx}$). It is thus indeed a translation $T_{h,W}(p)(x)=p(x+h)$ (whence the name). Hence the corresponding Appell sequences are those sequences $\{p_{n}(x)\}_{n\in\mathbb{N}}$ which satisfy $p_{n}'=np_{n-1}$ (the $\lambda$s cancel). The Sheffer sequences are defined by the operator $Q_{A}$ which sends $p_{n}$ to $\lambda np_{n-1}$ commuting with replacing the argument $x$ by $x+h$ (or just with $\frac{d}{dx}$), and the binomial sequences satisfy $p_{n}(x+h)=\sum_{k=0}^{n}\binom{n}{k}p_{n-k}(h)p_{k}(x)$ (indeed, a binomial relation). This is the only case where the composition $T_{g,W} \circ T_{h,W}$ gives another translation $T_{g+h,W}$, and the product of two evaluation functionals $\varepsilon_{h}$ and $\varepsilon_{h}$ is also an evaluation functional $\varepsilon_{g+h}$, for any $g$ and $h$.

\smallskip

\noindent\textbf{Example: Geometric Series}. Now let $\mathbb{F}$ be arbitrary, fix again such $\lambda$, and consider the case where $w_{n}=\lambda^{n}$. The series $W(t)$ here equals $\frac{\lambda}{\lambda-t}$, and the operator $D_{W}$, whose action sends $\frac{x^{n}}{\lambda^{n}}$ to $\frac{x^{n-1}}{\lambda^{n-1}}$, is given by $p\mapsto\lambda\frac{p(x)-p(0)}{x}$. The formula for $T_{h,W}(x^{n})$ is $\sum_{k=0}^{n}h^{n-k}x^{k}=\frac{x^{n+1}-h^{n+1}}{x-h}$ here. As $D_{W}$ is some normalization of $\frac{\lambda}{x}$, we find see that $T_{h,W}$ roughly multiplies each polynomial by $\frac{1}{1-h/x}=\frac{x}{x-h}$: Indeed, the exact formula for the action of $T_{h,W}$ is given by $T_{h,W}(p)(x)=\frac{xp(x)-hp(h)}{x-h}$ in this case. In addition, composing $T_{h,W}$ with $D_{W}$ is equivalent to subtracting the identity operator from it and dividing by $\frac{h}{\lambda}$. This operation sends a polynomial $p$ to $\widetilde{p}_{h}(x)=\lambda\frac{p(x)-p(h)}{x-h}$. A sequence $\{p_{n}(x)\}_{n\in\mathbb{N}}$ is Appell in this context if it satisfies the condition $p_{n}(x)-p_{n}(0)=xp_{n-1}(x)$. For Sheffer sequences we require the condition that for every $h\in\mathbb{F}$ the operator $Q_{A}$ taking $p_{n}$ to $\lambda p_{n-1}$ also takes $\widetilde{p}_{n,h}$ to $\lambda\widetilde{p}_{n-1,h}$. The binomial sequences in this setting are those which satisfy $p_{0}(x)=1$, $p_{n}(0)=0$ for all $n\geq1$, and the equality $\frac{p_{n}(x)}{x}-\frac{p_{n}(h)}{h}=(x-h)\sum_{k=1}^{n-1}\frac{p_{n-k}(h)}{h}\cdot\frac{p_{k}(x)}{x}$ for every $n\geq1$.

We remark that the Sheffer sequences in this case are characterized in \cite{[LM]} as those sequences which satisfy a recursion relation of the sort which we write as $\sum_{k=0}^{n}g_{n-k}p_{k}(x)=xp_{n-1}(x)+f_{n}$ for $n\geq0$ (with $p_{-1}=0$) with two invertible power series $f(y)=\sum_{n=0}^{\infty}f_{n}y^{n}$ and $g(y)=\sum_{n=0}^{\infty}g_{n}y^{n}$ from $\mathbb{F}\ldbrack y \rdbrack$. When checking what happens with this property for general $W$, one sees that conjugating the appropriate matrices by the matrix $U$ from Corollary \ref{conjW} and Proposition \ref{conjpres} does not give nice recursive relations in general. Indeed, the resulting relations involve the operator $D_{W}$, so that in the exponential case, for example, we get differential equations for the sequence. This is also related to the fact that the series $g(y)$ and $f(y)$ are $\frac{y}{\beta(y)}$ and $\frac{y\alpha(y)}{\beta(y)}$ using the parameters of $A \in R_{W}$ from Proposition \ref{genfunc}: With our choice of $W$ the expression $\alpha(y)W\big(x\beta(y)\big)=\frac{\alpha(y)}{1-x\beta(y)}$ can be easily written as $\frac{f(y)}{g(y)-xy}$ using the parameters $f$ and $g$, a property which is lost for other choices of $W$ (since $\alpha(y)=\frac{f(y)}{g(t)}$ and $\beta(y)=\frac{y}{g(y)}$ involve the inverse of $g$ in general).

\smallskip

The choice of $w_{n}=\big(\frac{\lambda}{1-q}\big)^{n}\prod_{j=1}^{n}(1-q^{j})$, with $q\in\mathbb{F}$ which is neither 0 nor a root of unity (e.g., some number which is transcendental over the prime field of $\mathbb{F}$) gives the \emph{$q$-umbral calculus}. Here $D_{W}(p)(x)=\lambda\frac{p(qx)-p(x)}{qx-x}$ is the $q$-derivative multiplied by $\lambda$, the equality which characterizes $W$-Appell sequences is $p_{n}(qx)=p_{n}(x)+x(q^{n}-1)p_{n-1}(x)$, and $T_{h,W}(p)(x)$ is described by the formula $\sum_{l=0}^{\infty}\sum_{k=0}^{l}\binom{l}{k}_{q}\frac{(-1)^{l-k}h^{l}p(q^{k}x)}{(qx-x)^{l}}q^{-[kl-\binom{k+1}{2}]}$, where $\binom{l}{k}_{q}$ is the $q$-binomial coefficient $\prod_{j=k+1}^{l}(1-q^{j})/\prod_{j=1}^{l-k}(1-q^{j})$. In this case the sequence of polynomials whose associated matrix represents $T_{h,W}$ take the form $p_{n}(x)=\prod_{j=0}^{n-1}(x+hq^{j})$, i.e., the roots of $p_{n}$ are the first $n$ terms of a geometric sequence with quotient $q$. As neither the formula for $T_{h,W}$ nor the other related ones seem to reduce to succinct expressions, we do not follow the detail of this example here. We just remark that when $\mathbb{F}$ is of characteristic 0 then the limit $q\to1$ exists (at least formally), yielding the exponential example from above.

\smallskip

The behavior of these examples with respect to the parameter $\lambda$ illustrate the operation of multiplying $w_{n}$ by $\lambda^{n}$ in the general case. This is given in the following
\begin{prop}
Let $W(t)=\sum_{n=0}^{\infty}\frac{t^{n}}{w_{n}}\in\mathbb{F}\ldbrack t \rdbrack$ and a non-zero element $\lambda\in\mathbb{F}$ be given. Then multiplying each $w_{n}$ by $\lambda^{n}$ leaves $R_{W}$, as well as its subgroups of Appell and binomial sequences, invariant. \label{rescale}
\end{prop}

\begin{proof}
This operation multiplies $D_{W}$, as well as $Q_{A,W}$ for any $A \in L$, by $\lambda$, and replaces $W(t)$ by $W\big(\frac{t}{\lambda}\big)$. Hence $T_{h,W}$ remains the same. This immediately proves the assertion for the Appell and Sheffer sequences. For sequences of binomial type this follows either from Definition \ref{Appbin} and the invariance of the quotient $\frac{w_{n}}{w_{n-k}w_{k}}$ under this operation, or from Proposition \ref{binalpha1}. This proves the proposition.
\end{proof}
Note that the parameters $\alpha$ and $\beta$ of a Sheffer sequence are not preserved by the operation from Proposition \ref{rescale}. Indeed, Proposition \ref{conjpres} shows that these parameters will be preserved if we conjugate by the diagonal matrix containing the powers of $\lambda$ on the diagonal. With no such conjugation, the matrix whose parameters with $W$ are $\alpha$ and $\beta$ will have, with the new weight $\widetilde{W}$ satisfying $W(t)=\widetilde{W}(\lambda t)$, the parameters $\alpha\big(\frac{y}{\lambda}\big)$ and $\beta\big(\frac{y}{\lambda}\big)$.

\smallskip

We remark that Theorem \ref{actdesc} allows us to extend the definitions of $R_{W}$, as well as its Appell and binomial subgroups, to the case where  $D_{W}$ is replaced by an arbitrary degree decreasing operator. Indeed, the Appell subgroup is the stabilizer of $M$, the group of corresponding Riordan arrays is its normalizer, and the binomial subgroup consists of those matrices $W$ in the latter group in which $a_{n,0}=\delta_{n,0}$. One example for an interesting such operator is, in characteristic 0, the \emph{finite difference operator} sending a polynomial $p$ to $p(x+a)-p(x)$ (in fact, weighted finite difference operators of the sort taking $p$ to $T_{a,W}(p)-p$, in any characteristic, also produce such matrices). Moreover, Proposition \ref{acttrans} shows that all these groups are again conjugate (hence isomorphic), where for preserving the group of sequences of binomial type we restrict attention to conjugators also satisfying the condition $a_{n,0}=\delta_{n,0}$. Even though these more general Riordan arrays represent sequences of polynomials satisfying more complicated relations, it is interesting to know that these groups are all algebraically isomorphic. The question whether every element of $L$ lies in a subgroup of such more general Riordan arrays is also worth investigating.

\smallskip

The fact that a sequence of polynomials is a Sheffer sequence for some weight $W(t)=\sum_{n=0}^{\infty}\frac{t^{n}}{w_{n}}\in\mathbb{F}\ldbrack t \rdbrack$ yields a lot of combinatorial information about the sequence. Hence it may be worthwhile to find when such a sequence is Sheffer for \emph{two different} such weights. For this we recall the \emph{extended binomial coefficients}, defined over any field $\mathbb{F}$ of characteristic 0, by noting that the expression $\frac{1}{n!}\prod_{j=0}^{n-1}(\xi-j)$ for the binomial coefficient $\binom{\xi}{n}$ makes sense for $\xi\in\mathbb{F}$. We shall make use of the formula given in the following
\begin{lem}
Given $\xi$ and $\eta$ in $\mathbb{F}$ and $n\in\mathbb{N}$, we have $\sum_{r+s=n}\binom{\xi}{r}\binom{\eta}{s}=\binom{\xi+\eta}{n}$. \label{extbin} \end{lem}

\begin{proof}
We apply induction on $n$. The case $n=0$ is trivial. Assume that the equality holds for $n$, and write $\binom{\xi+\eta}{n+1}$ as $\frac{\xi+\eta-n}{n+1}\binom{\xi+\eta}{n}$. The induction hypothesis allows us to write the latter expression as \[\frac{\xi+\eta-n}{n+1}\sum_{r+s=n}\binom{\xi}{r}\binom{\eta}{s}=\sum_{r+s=n}\bigg[\frac{\xi-r}{n+1}+\frac{\eta-s}{n+1}\bigg]\binom{\xi}{r}\binom{\eta}{s}=\]
\[=\sum_{\substack{\tilde{r}+s=n+1 \\ \tilde{r}>0}}\frac{\tilde{r}}{n+1}\binom{\xi}{\tilde{r}}\binom{\eta}{s}+\sum_{\substack{r+\tilde{s}=n+1 \\ \tilde{s}>0}}\frac{\tilde{s}}{n+1}\binom{\xi}{r}\binom{\eta}{\tilde{s}},\] where we have set $\tilde{r}=r+1$ and $\tilde{s}=s+1$ respectively. We may include the values $\tilde{r}=0$ and $\tilde{s}=0$ because the corresponding summands vanish, and after merging the sums and using the equality $\frac{\tilde{r}}{n+1}+\frac{\tilde{s}}{n+1}=1$ we get the desired expression $\sum_{r+s=n+1}\binom{\xi}{r}\binom{\eta}{s}$. This proves the lemma.
\end{proof}
Putting $\eta=1$ in Lemma \ref{extbin} and noting that $\binom{1}{s}$ is 1 for $s=0$ and $s=1$ and 0 for $s>1$ yields
\begin{cor}
The equality $\binom{\xi+1}{l}=\binom{\xi}{l}+\binom{\xi}{l-1}$ holds for any
$l\geq1$ and $\xi\in\mathbb{F}$. \label{binpm1}
\end{cor}
It is clear that the expressions in Lemma \ref{extbin} and Corollary \ref{binpm1} are defined, and the assertions hold, not only in characteristic 0, but wherever the characteristic of $\mathbb{F}$ does not divide any of the factorials appearing in the denominators.

Now, if one sequence is Appell then we have the following
\begin{thm}
Take $W$ as above, and let $A$ be the element of $R_{W}$ representing a $W$-Appell sequence with parameter $\alpha(y)=\sum_{n=0}^{\infty}c_{n}y^{n}$, where $c_{0}\neq0$. This matrix lies in $R_{\widetilde{W}}$ for some weight $\widetilde{W}(t)=\sum_{n=0}^{\infty}\frac{t^{n}}{\widetilde{w}_{n}}\in\mathbb{F}\ldbrack t \rdbrack$ in the following two cases: $(i)$ The expression $\gamma_{k}=\frac{\widetilde{w}_{k}w_{k+1}}{\widetilde{w}_{k+1}w_{k}}$ is a non-zero constant, independent of $k$. $(ii)$ $\mathbb{F}$ has characteristic 0, $\gamma_{k}$ is a non-constant, never vanishing linear function of $k$, and $\alpha(y)=c_{0}e^{hy}$ for some $h\in\mathbb{F}$. Conversely, if $c_{1}\neq0$ then $A \in R_{\widetilde{W}}$ only if one of the conditions $(i)$ or $(ii)$ is satisfied. \label{AppShef}
\end{thm}

\begin{proof}
From the equality
$\sum_{n=0}^{\infty}\sum_{k=0}^{n}a_{n,k}\frac{x^{k}y^{n}}{w_{n}}=\sum_{l=0}^{\infty}c_{l}\sum_{k=0}^{\infty}\frac{x^{k}y^{k+l}}{w_{k}}$ from the proof of Proposition \ref{Appbeta1}, we find that $a_{n,k}=\frac{w_{n}c_{n-k}}{w_{k}}$ for every $n \geq k$. Given another weight
$\widetilde{W}$, we thus have
$C_{A_{k},\widetilde{W}}(y)=\sum_{n=0}^{\infty}\frac{w_{n}c_{n-k}y^{n}}{w_{k}\widetilde{w}_{n}}$. To see whether the condition from Definition \ref{Riordan} is satisfied, we have to check when does the equality
$\widetilde{w}_{k}^{2}C_{A_{k},\widetilde{W}}^{2}=\widetilde{w}_{k-1}C_{A_{k-1},\widetilde{W}}\widetilde{w}_{k+1}C_{A_{k+1},\widetilde{W}}$ hold for every $k\geq1$. Explicitly, this equality becomes \[\frac{\widetilde{w}_{k}^{2}}{w_{k}^{2}}\sum_{n,m}\frac{w_{n}w_{m}c_{n-k}c_{m-k}y^{n+m}}{\widetilde{w}_{m}\widetilde{w}_{n}}=
\frac{\widetilde{w}_{k+1}\widetilde{w}_{k-1}}{w_{k-1}w_{k+1}}\sum_{n,m}\frac{w_{n}w_{m}c_{n-1-k}c_{m+1-k}y^{n+m}}{\widetilde{w}_{m}\widetilde{w}_{n}},\] where on the left hand side both indices $m$ and $n$ start from $k$ while on the right hand side $m$ starts from $k-1$ while $n$ starts from $k+1$. Given a number $p\geq2k$, the coefficients of $y^{p}$ are
\[\sum_{n+m=p}c_{n-k}c_{m-k}\prod_{j=k}^{n-1}\gamma_{j}\prod_{j=k}^{m-1}\gamma_{j}\quad\mathrm{and}\quad
\sum_{n+m=p}c_{n-1-k}c_{m+1-k}\prod_{j=k+1}^{n-1}\gamma_{j}\prod_{j=k-1}^{m-1}\gamma_{j},\] where the limitation on $m$ and $n$ are as above and empty products equal 1. Now, if $\gamma_{j}$ is some non-zero constant $\lambda$, regardless of $j$, then this equality clearly holds for any $\alpha$, so that $A \in R_{\widetilde{W}}$ in case $(i)$ (indeed, this condition on $\gamma_{j}$ is equivalent to $\widetilde{W}$ being related to $W$ in the manner described in Proposition \ref{rescale}, so that the assertion in this case is a consequence of that proposition). On the other hand, if $\mathbb{F}$ has characteristic 0 and $\gamma_{j}=\lambda-\sigma j$ for some $\lambda$ and $\sigma$ from $\mathbb{F}$ with $\sigma\neq0$ then any product of the form $\prod_{j=\kappa}^{\nu-1}\gamma_{j}$ for some integers $\kappa$ and $\nu$ with $\nu\geq\kappa$ can be written as $\sigma^{\nu-\kappa}(\nu-\kappa)!\binom{\lambda/\sigma-\kappa}{\nu-\kappa}$. But our assumption on $\alpha$ means that $c_{l}=\frac{c_{0}h^{l}}{l!}$ for every $l$, so that the two sums we have to compare amount to $c_{0}^{2}(h\sigma)^{p-2k}$ times $\sum_{n+m=p}\binom{\lambda/\sigma-k}{n-k}\binom{\lambda/\sigma-k}{n-k}$ and
$\sum_{n+m=p}\binom{\lambda/\sigma-1-k}{n-1-k}\binom{\lambda/\sigma+1-k}{m+1-k}$ respectively. But after the appropriate translations of indices Lemma \ref{extbin} shows
that the two latter sums both equal $\binom{2\lambda/\sigma-2k}{p-2k}$. Hence
$A \in R_{\widetilde{W}}$ also under condition $(ii)$.

Conversely, we first note that the valuation of both power series is $2k$, and
that both start with $c_{0}^{2}y^{2k}$ and continue with higher powers of $y$. The coefficient of $y^{2k+1}$ is $2c_{0}c_{1}\gamma_{k}$ on the left hand side, while on the right hand side we get $c_{0}c_{1}(\gamma_{k+1}+\gamma_{k-1})$. Assuming that $c_{1}\neq0$, we find that the difference $\gamma_{k+1}-\gamma_{k}$ must be independent of $k$, so that $\gamma_{k}$ must be of the form $\lambda-\sigma k$ where $\lambda$ and $\sigma$ are constants which are independent of $k$. Now, if $\sigma=0$ then $\lambda\neq0$ (since none of the $\gamma_{k}$ may vanish), and we are in case $(i)$. Assuming that $\sigma\neq0$, it remains to prove that $\mathbb{F}$ has characteristic 0 and that $\alpha$ has the desired form. Let $h=\frac{c_{1}}{c_{0}}\neq0$, and we write each $c_{l}$ as $c_{0}h^{l}d_{l}$ for some $d_{l}\in\mathbb{F}$. We choose some $k$, and write the equation between the coefficients of $y^{p}$ as follows: Take $p=2k+l$, let $\mu=\frac{\lambda}{\sigma}-k$, and make the index change $n=r+k$ and $m=s+k$ on the left hand side while writing $n=r+k+1$ and $m=s+k-1$ on the right hand side. Substituting the value of $\gamma_{j}$ and dividing the resulting equality by $c_{0}^{2}(\sigma h)^{p-2k=l}$ produces the equation
\[\sum_{r+s=l}d_{r}d_{s}\prod_{j=0}^{r-1}(\mu-j)\prod_{j=0}^{s-1}(\mu-j)=\sum_{r+s=l}d_{r}d_{s}\prod_{j=0}^{r-1}(\mu-1-j)\prod_{j=0}^{s-1}(\mu+1-j)\] (which is independent of $k$). The equalities for $l=0$ and $l=1$ are tautological, and $d_{0}=d_{1}=1$ by definition. We now use the equation for larger $l$ to prove that $d_{l}$ satisfies the equality $l!d_{l}=1$ in $\mathbb{F}$, so that $l!$ is invertible in $\mathbb{F}$ and $d_{l}=\frac{1}{l!}$. Indeed, the induction hypothesis allows us to write $d_{r}d_{s}=\frac{1}{r!s!}$ unless $r=l$ and $s=0$ or the other way around, with the characteristic of $\mathbb{F}$ allowing this. Our equality then becomes
\[2d_{l}\prod_{j=0}^{l-1}(\mu-j)+\!\sum_{\substack{r+s=l \\ rs\neq0}}\!\binom{\mu}{r}\binom{\mu}{s}\!=\!d_{l}\!\!\prod_{j=-1}^{l-2}\!(\mu-j)+d_{l}\!\prod_{j=1}^{l}(\mu-j)+\!\sum_{\substack{r+s=l \\ rs\neq0}}\!\binom{\mu+\!1}{r}\binom{\mu-\!1}{s}\!.\] The difference $\prod_{j=-1}^{l-2}(\mu-j)-\prod_{j=0}^{l-1}(\mu-j)$ is $l\prod_{j=0}^{l-2}(\mu-j)=l!\binom{\mu}{l-1}$ (we already know that $(l-1)!$ is invertible in $\mathbb{F}$), and the difference between $\prod_{j=0}^{l-1}(\mu-j)$ and $\prod_{j=1}^{l}(\mu-j)$ similarly equals $l\prod_{j=1}^{l}(\mu-j)=l!\binom{\mu-1}{l-1}$. Moreover, Corollary \ref{binpm1} (with our knowledge on the characteristic of $\mathbb{F}$) allows us to write $\binom{\mu+1}{r}$ as $\binom{\mu}{r}+\binom{\mu}{r-1}$ on the right hand side as well as $\binom{\mu}{s}=\binom{\mu-1}{s}+\binom{\mu-1}{s-1}$ on the left hand side. After cancellations we then obtain \[d_{l}l!\binom{\mu-1}{l-1}+\sum_{\substack{r+s=l \\ rs\neq0}}\binom{\mu}{r}\binom{\mu-1}{s-1}=d_{l}l!\binom{\mu}{l-1}+\sum_{\substack{r+s=l \\ rs\neq0}}\binom{\mu}{r-1}\binom{\mu-1}{s}.\] By Lemma \ref{extbin}, the sum on the left hand side equals $\binom{2\mu-1}{l-1}-\binom{\mu-1}{l-1}$, while the one on the right hand side is $\binom{2\mu-1}{l-1}-\binom{\mu}{l-1}$. Another application of Corollary \ref{binpm1} produces, after some additional cancellations, the equality $\binom{\mu-1}{l-2}=d_{l}l!\binom{\mu-1}{l-2}$, from which the desired equality follows since the non-vanishing of all of the numbers $\gamma_{k}$ implies that $\binom{\mu-1}{l-2}\neq0$. Hence $\mathbb{F}$ is of characteristic 0, $d_{l}=\frac{1}{l!}$ for all $l$, $c_{l}=c_{0}\frac{h^{l}}{l!}$ for all $l$, and $\alpha(y)=c_{0}e^{hy}$ as
required. This completes the proof of the theorem.
\end{proof}

\begin{cor}
Take $W(t)=e^{t}$ (in characteristic 0). Then a $W$-Appell sequence with $p_{1}(0)\neq0$ is $\widetilde{W}$-Sheffer for some weight $\widetilde{W}$ if and only if either $\widetilde{W}(t)=e^{t/\lambda}$ for some $0\neq\lambda\in\mathbb{F}$, or $\widetilde{W}(t)=(1+\sigma t)^{\lambda/\sigma}$ for some non-zero $\sigma\in\mathbb{F}$ and $\lambda\in\mathbb{F}\setminus\mathbb{N}\sigma$ and the associated matrix $A \in R_{W} \subseteq L$ from Proposition \ref{seqmat} is a non-zero scalar multiple of the matrix from Corollary \ref{ThWApp} which represents a $W$-translation $T_{h,W}$ for some $h\in\mathbb{F}$. \label{ThWtildeW}
\end{cor}

\begin{proof}
We just substitute $W(t)=e^{t}$ in Theorem \ref{AppShef}. The case $\widetilde{W}(t)=e^{t/\lambda}$ is case $(i)$ there (or Proposition \ref{rescale}). Otherwise the requirement that $\alpha(y)=c_{0}e^{hy}$ is precisely the condition from Corollary \ref{ThWApp} for our $W$ (up to a multiplicative scalar). Moreover, the condition $\gamma_{k}=\lambda-\sigma k$ yields here $\frac{1}{\widetilde{w}_{k+1}}=\frac{\lambda-\sigma k}{\widetilde{w}_{k}(k+1)}$. This shows, by a simple induction, that $\frac{1}{\widetilde{w}_{k}}$ is the coefficient $\sigma^{k}\binom{\lambda/\sigma}{k}$ of $(1+\sigma t)^{\lambda/\sigma}$. This proves the corollary.
\end{proof}

In fact, Proposition \ref{conjpres} allows us to translate the (simpler) assertion of Corollary \ref{ThWtildeW} to the general case considered in Theorem \ref{AppShef}.

We remark that in case $c_{1}=0$ the results of Theorem \ref{AppShef} and Corollary \ref{ThWtildeW} become more complicated. To give a rough idea of this, let $d=v\big(\alpha(y)-\alpha(0)\big)$ by the minimal index $k>0$ such that $c_{k}\neq0$. Then a more detailed analysis of the proof of Theorem \ref{AppShef} shows that the linear condition holds now not for the numbers $\gamma_{k}$ but for the products $\prod_{i=0}^{d-1}\gamma_{k+i}$. This allows us to obtain relations only between $\widetilde{w}_{k}$ and $\widetilde{w}_{k+d}$, but not directly between $\widetilde{w}_{k}$ and $\widetilde{w}_{k+1}$. Indeed, if $d=\infty$ then $\alpha(y)=c_{0}$ and $A$ is a scalar matrix, which lies in $R_{\widetilde{W}}$ for every $\widetilde{W}$. This illustrates the fact that as $d$ grows, less restrictions on $\widetilde{W}$ must be imposed for the $W$-Appell sequence to be $\widetilde{W}$-Sheffer.

\noindent\textsc{Einstein Institute of Mathematics, the Hebrew University of Jerusalem, Edmund Safra Campus, Jerusalem 9190401, Israel}

\noindent E-mail address: zemels@math.huji.ac.il

\end{document}